\documentclass[12pt]{amsart}
\usepackage{tikz}
\usepackage{caption}
\usepackage{subcaption}

\usepackage{float}

\usepackage{graphicx}
\usepackage{color}
\usepackage{amsmath}
\usepackage{amssymb}
\newcommand{\beq}{\begin{eqnarray*}}
\newcommand{\feq}{\end{eqnarray*}}
\newcommand{\beqn}{\begin{eqnarray}}
\newcommand{\feqn}{\end{eqnarray}}

\newtheorem{theorem}{Theorem}[section]
\newtheorem{lemma}[theorem]{Lemma}

\theoremstyle{definition}

\theoremstyle{remark}

\numberwithin{equation}{section}

\allowdisplaybreaks

\begin{document}
\title[EP zero background]{On the Riccati dynamics of the Euler-Poisson equations with zero background state}
\author{Yongki Lee}
\address{Department of Mathematical Sciences, Georgia Southern University, Statesboro,  30458}
\email{yongkilee@georgiasouthern.edu}
\begin{abstract} 
This paper studies the two-dimensional Euler-Poisson equations associated with either attractive or repulsive forces. We mainly study the Riccati system that governs the 
flow's gradient. Under a suitable condition, it is shown that the Euler-Poisson system admits global smooth solutions for a large set of initial configurations. This paper is a continuation of our former work \cite{Lee20}.
\end{abstract}
\maketitle


\section{Introduction}

We are concerned with the threshold phenomenon in two-dimensional Euler-Poisson (EP) equations. The pressureless Euler-Poisson equations in multi-dimensions are
\begin{equation}\label{101}
\left\{
  \begin{array}{ll}
        \rho_t + \nabla \cdot (\rho \mathbf{u})=0, \\
\mathbf{u}_t + \mathbf{u} \cdot \nabla \mathbf{u} = k \nabla \Delta^{-1} (\rho-c_b),\\
  \end{array}
\right.
\end{equation}
which are the usual statements of the conservation of mass and Newton's second law. Here $k$ is a physical constant which parameterizes the repulsive $k>0$ or attractive $k<0$ forcing, governed by the Poisson potential $\Delta^{-1}(\rho -c_b)$ with constant $c_b\geq 0$ which denotes background state.
  The density $\rho=\rho(t,x)$ : $\mathbb{R}^+ \times \mathbb{R}^n \mapsto \mathbb{R^+} $  and the velocity field $\mathbf{u}(t,x)$ : $\mathbb{R}^+ \times \mathbb{R}^n \mapsto \mathbb{R}^{n}$  are the unknowns. This hyperbolic system \eqref{101} with non-local forcing describes the dynamic behavior of many important physical flows, including  plasma with collision, cosmological waves, charge transport,  and the collapse of stars due to self gravitation \cite{BRR94, DLYZ02, HJL81, Ma86, MRS90}.

This paper continues to investigate the Riccati dynamics of the Euler-Poisson equations in the  preceding manuscript \cite{Lee20}, which studied  a sub-threshold conditions for  global smooth solutions. In the preceding paper,  we studied threshold conditions for global existence of solutions to the EP equations with attractive forcing and non-zero background state case.

The present work is devoted to study both attractive and repulsive force cases with zero background state. More precisely, the goal of this paper is showing that, under a suitable condition, two-dimensional Euler-Poisson system with zero background (i.e., $c_b=0$) can afford to have global smooth solutions for a large set of initial configurations. It is interesting to note that the sub-threshold condition for global existence we find in this work holds for both attractive and repulsive forcing cases. 


\section{Problem formulation and main result}\label{section 2}
We consider two-dimensional Euler-Poisson equations with either attractive ($k<0$) or repulsive ($k>0$) forces under the zero background state ($c_b=0$). In \cite{Lee20}, the author studied EP system with attractive forcing under the non-zero background case. 

We are mainly concerned with a Riccati equation that governs $\mathcal{M}:=\nabla \mathbf{u}$, and it is obtained by differentiating the second equation of \eqref{101}:
\begin{equation}\label{matrix_eqn}
\partial_t \mathcal{M} + \mathbf{u} \cdot \nabla \mathcal{M} + \mathcal{M}^2 = k \nabla \otimes \nabla \Delta^{-1} \rho  = k R[\rho -c_b],
\end{equation}
where $R[\cdot] $ is the $2 \times 2$ Riesz matrix operator, defined as
$$R[h]:=\nabla \otimes \nabla \Delta^{-1}[h]=\mathcal{F}^{-1}\bigg{\{} \frac{\xi_i \xi_j}{|\xi|^2} \hat{h}(\xi)  \bigg{\}}_{i,j=1,2} .$$

The problem formulation is similar to the ones in \cite{Y16, Lee20}, so we omit several details, but briefly discuss it here. 

There are several quantities that characterize the dynamics of $\mathcal{M}$.  These are the trace $d:=\mathrm{tr} \mathcal{M} = \nabla \cdot \mathbf{u}$, the vorticity $\omega : = \nabla \times \mathbf{u} = \mathcal{M}_{21} - \mathcal{M}_{12}$ and quantities $\eta:= \mathcal{M}_{11} - \mathcal{M}_{22}$ and $\xi := \mathcal{M}_{12} + \mathcal{M}_{21}$. The trace is governed by the Riccati equation:
\begin{equation}\label{riccati_d}
d'=-\frac{1}{2}d^2 -\frac{1}{2}\eta^2 + \frac{1}{2}\omega^2 -\frac{1}{2}\xi^2 + k(\rho-c_b),
\end{equation}
where $[\cdot]' = \partial/\partial t + u \cdot \nabla$ is the derivative along the characteristic path.  

Each terms in \eqref{riccati_d} satisfies nonlinear ODEs along the same characteristics. These are
\begin{equation}\label{omega_rho}
\omega' + \omega d =0 \ and   \ \rho' + \rho d =0.
\end{equation}
In contrast to these, $\eta$ and $\xi$ are governed non-local ODEs
\begin{equation}\label{eta_xi}
\eta' +\eta d =k(R_{11}[\rho-c_b] - R_{22}[\rho-c_b]) \ and \      \xi' + \xi d = k (R_{12}[\rho-c_b]+R_{21}[\rho-c_b]).
\end{equation}
Here, $R_{ij}[\cdot]$ can be explicitly written as follows:
\begin{equation}\label{R_ij}
(R_{i j}  [h])(x)  = p.v.\frac{1}{2\pi}\int_{\mathbb{R}^2} \frac{\partial^2}{\partial y_j \partial y_i}\log |y| \cdot h(x - y) \, d y +  \frac{h(x)}{2\pi}\int_{|z| =1}   z_i z_j  \, dz,
\end{equation}

Now, \eqref{riccati_d}-\eqref{R_ij} allow us to obtain the closed ODE system:
\begin{equation}\label{ode2_intro}
\left\{
  \begin{array}{ll}
    d' = -\frac{1}{2}d^2 + A(t)\rho^2 +k(\rho-c_b), \\
    \rho' = -\rho d,\\
  \end{array}
\right.
\end{equation}
subject to initial data $(\nabla \mathbf{u}, \rho)(0,\cdot)=(\nabla \mathbf{u}_0, \rho_0),$
where
\begin{equation}\label{A_eqn_2}
A(t):=\frac{1}{2} \bigg{[} \bigg{(} \frac{\omega_0}{\rho_0} \bigg{)}^2 - \bigg{(} \frac{n_0}{\rho_0} + \int^t _0 \frac{f_1(\tau)}{\rho(\tau)} \, d \tau \bigg{)}^2 - \bigg{(} \frac{\xi_0}{\rho_0} + \int^t _0 \frac{f_2(\tau)}{\rho(\tau)} \, d \tau \bigg{)}^2  \bigg{]}.
\end{equation}
Here, $f_i$ are singular integrals inherited from the Riesz transform:
\begin{equation*}
f_1(t):=k(R_{11}[\rho-c_b] - R_{22}[\rho-c_b]) = \frac{k}{\pi} p.v.\int_{\mathbb{R}^2} \frac{-y^2 _1 + y^2 _2}{(y^2 _1 + y^2 _2)^2} \rho(t, x(t) - y ) \, dy,
\end{equation*}
and
\begin{equation*}
f_2(t):=k (R_{12}[\rho-c_b]+R_{21}[\rho-c_b]) = \frac{k}{\pi} p.v.\int_{\mathbb{R}^2} \frac{-2y_1 y_2}{(y^2 _1 + y^2 _2)^2} \rho(t, x(t) - y ) \, dy.
\end{equation*}
Detailed derivations of  \eqref{riccati_d}-\eqref{A_eqn_2} can be found in \cite{Y16, Lee20}. We note that all functions of consideration here are evaluated along the characteristic, that is, for example,
$f_i (t)=f_i(t, x(t))$ and $\eta(t)=\eta(t, x(t))$, etc.

From \eqref{A_eqn_2}, we see that $A(t)$ is uniformly bounded above. That is
 $$A(t) \leq \frac{1}{2} \bigg{(}\frac{\omega_ 0}{\rho_0} \bigg{)}^2 .$$ However, an existence of any lower bound for $A(t)$ is generally not known.
 The growth of $|A(t)|$ is related to the density's Riesz transform and non-local/singular nature of this makes it difficult to study the dynamics of system  \eqref{ode2_intro}.  $|A(t)|$ may blows-up in finite time. In this case, $(\rho(t), d(t)) \rightarrow (\infty, -\infty)$ in finite time (see \cite{Lee20}).  Thus, in the scope of a global regularity, it remains the case that $|A(t)|$ blows-up at infinity or is bounded uniformly. That is, we investigate the case that there exists some function $h(t)$ such that $A(t) \geq h(t)$, for all $t\geq 0$ and
$$h(t) \rightarrow -\infty \ as \ t \rightarrow \infty.$$
More discussions about the motivation of this study can be found in \cite{Lee20}. Furthermore, it turns out that many of so-called restricted/modified models \cite{Y16, Y17, LT03, Tan14} and the results under the vanishing initial vorticity assumption \cite{CT08, CT09} can be reinterpreted using our proposed structure \eqref{ode2_intro}-\eqref{A_eqn_2}.

In \cite{Lee20}, when $k<0$ and $c_b\neq 0$, it has been studied that the nonlinear-nonlocal system  \eqref{ode2_intro} admit global smooth solutions for a large set of initial configurations provided that
\begin{equation}\label{exp11}
A(t) \geq -\alpha_1 e^{\beta_1 t}, \ \ for \ all \ t,
\end{equation}
where $\alpha_1$ and $\beta_1$ are some positive constants. In other words, the system   \eqref{ode2_intro} can afford to have global solutions even if 
$A(t) \rightarrow -\infty, \ as \ t \rightarrow \infty$,
 as long as the blow-up rate of $|A(t)|$ is not higher than that of an exponential.

 In this work, we consider \eqref{ode2_intro} with $c_b =0$, and show similar results under the condition that
 \begin{equation}\label{poly_condi_33}
 A(t) \geq - (\alpha t + \beta)^s, \ for \ all \ t,
 \end{equation}
 where $\alpha, \beta >0$ and $s \geq 1$.
 In contrast to $c_b >0$ case, it turns out that when $c_b =0$, $(\rho, d)$  may blow up for \emph{all} initial data with the assumption in \eqref{exp11}. Thus, it is interesting to note that $c_b$ serves as a key factor that  distinguishes the dynamics of \eqref{ode2_intro}.


For notational convenience, in the rest of paper, we assume that $k=\pm 1$ and $\alpha=\beta=1$, because these constants are not essential in our analysis. Thus, we assume that
\begin{equation}\label{poly_condi_22}
A(t) \geq - ( t + 1)^s, \ for \ all \ t.
\end{equation}
We note that the above inequality already assumes that $A(0)\geq -1$, that is,
$$A(0)=\frac{1}{2} \bigg{[} \bigg{(} \frac{\omega_0}{\rho_0} \bigg{)}^2 - \bigg{(} \frac{n_0}{\rho_0}  \bigg{)}^2 - \bigg{(} \frac{\xi_0}{\rho_0}  \bigg{)}^2  \bigg{]}\geq -1.
$$
But this does not restrict our result, since one can always find $\beta$ that satisfy \eqref{poly_condi_33} for any $A(0)$.

To present our results we write \eqref{matrix_eqn} and the second equation of \eqref{omega_rho}, to establish the two-dimensional Euler-Poisson system:
\begin{equation}\label{EP_system11}
\left\{
  \begin{array}{ll}
    \mathcal{M}' + \mathcal{M}^2 =kR[\rho ], \\
    \rho' = -\rho\cdot \mathrm{tr}(\mathcal{M}),\\
  \end{array}
\right.
\end{equation}
subject to initial data $(\mathcal{M},\rho)(0,\cdot)=(\mathcal{M}_0 , \rho_0)$.

Our goal of this work is to prove the following result.
\begin{theorem}\label{thm_poly}
Consider the Euler-Poisson system, \eqref{EP_system11}-\eqref{poly_condi_22} with either $k=1$ or $k=-1$. 
If $$(\rho_0 , d_0) \in \Omega_{ s}:= \{(\rho ,d) \in \mathbb{R}^2  |  \rho>0 , d >0, \ and \  d > m^* \rho + n^*      \},$$
 then the solution of the Euler-Poisson system remains smooth for all time.
 Here, $m^* = m_1 m^M _2$ and $n^* = n_1 n _2$ are constants satisfying
\begin{equation}\label{condi_parta}
M\geq \max\{s/2 ,1 \}, \ m_1 > \sqrt{2}, \ m_2 > \max \{n_2 ,1 \}, \ n_2 >0,
  \end{equation}
and
\begin{equation*}
n_1 > \max{ \bigg{\{}  \frac{(m_2 + m_1 M m^M _2)^2}{2(m^2 _1 -2)m^{2M}_2} + \frac{1}{2}, 1  \bigg{\}} }.
\end{equation*}
\end{theorem}

\textbf{Remarks:} Some remarks are in order.

(1) We first note that  the sub-threshold condition for global existence in the theorem works for both attractive and repulsive forcing cases.

(2) For simplicity, we set $\alpha=\beta=1$ in \eqref{poly_condi_33} and $k=\pm 1$ in \eqref{EP_system11}. Our method works equally well for general constants.

(3) As discussed in \cite{Lee20}, finite time blow-up of $|A(t)|$ leads to finite time blow-up of solutions for \eqref{EP_system11}. The main contribution of the theorem is that the Riccati structure \eqref{EP_system11} affords to have smooth solutions while $A(t)$ freely moves under the condition
$$A(t) \geq -(\alpha t + \beta)^s, \ \ for \ all \ t.$$
 In particular, the system admits global solutions even though $A(t)$ blows-up at infinity, as long as $A(t) \rightarrow -\infty$ with a polynomial blow-up rate at any order.

(4) The proof of theorem is based on the method proposed in the preceding paper \cite{Lee20}. That is, we first construct an auxiliary system in $3$D space and find an invariant space of the system. The invariant space, where all trajectories if they start from inside this space will stay encompassed at all time. The projection of the $3$D invariant space onto $2$D space will serve as an invariant region of system  \eqref{EP_system11}.  The key parts of the proofs are constructing the surface that determines invariant space of the auxiliary system, and establishing monotonicity relation between the auxiliary system and the original system.

(5) Consider \eqref{ode2_intro}. There are two aspects that distinguish $c_b=0$ case from $c_b \neq 0$ for threshold analysis. First, the solutions in $c_b=0$ case blow-up in finite time for all initial data if
$$A(t) \sim -e^t$$
as opposed to $c_b\neq 0$ case can afford to have global smooth solutions. We refer to the Appendix section of this paper for more discussion about this.
Secondly, when $c_b =0$ and $k>0$, there may exist oscillatory solutions $(\rho(t), d(t))$, with arbitrarily large amplitudes, see Figure \ref{fig111} (note that the equilibrium point is $(\frac{1}{A(t)}, 0)$). This makes difficult to `trap' a trajectory because as $A(t)$ (which we do not know exact behavior) changes, the oscillatory type orbit may change to blow-up type trajectory, or the other way around too. In contrast to this, when $c_b \neq 0$, $(\rho(t), d(t))$ are either blow-up in finite time or globally exist for each fixed values of $A(t)$. 
This difference leads to the difference in the construction of invariant space, comparing to the one in \cite{Lee20}. See Figure \ref{fig1}.









\begin{figure}[h]
\begin{subfigure}{.85\textwidth}
  \centering
  \includegraphics[width=1\linewidth]{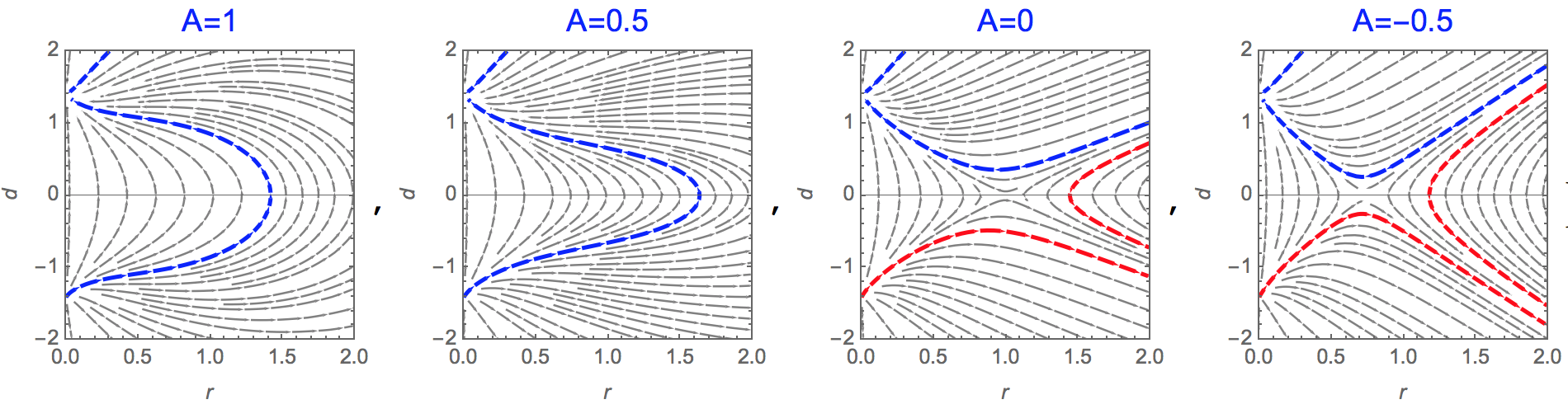}
  \caption{$c_b=1$ and $k=-1$ }
  \label{fig_att:sfig1}
\end{subfigure}%

\begin{subfigure}{.85\textwidth}
  \centering
  \includegraphics[width=1\linewidth]{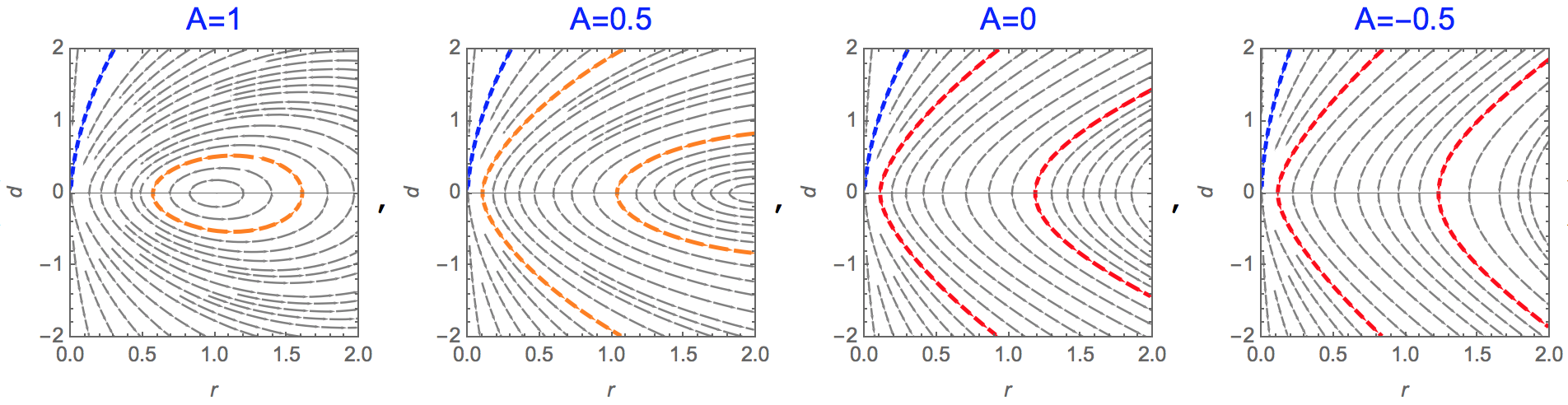}
  \caption{$c_b=0$ and $k=-1$ }
  \label{fig_att:sfig2}
\end{subfigure}

\begin{subfigure}{.85\textwidth}
  \centering
  \includegraphics[width=1\linewidth]{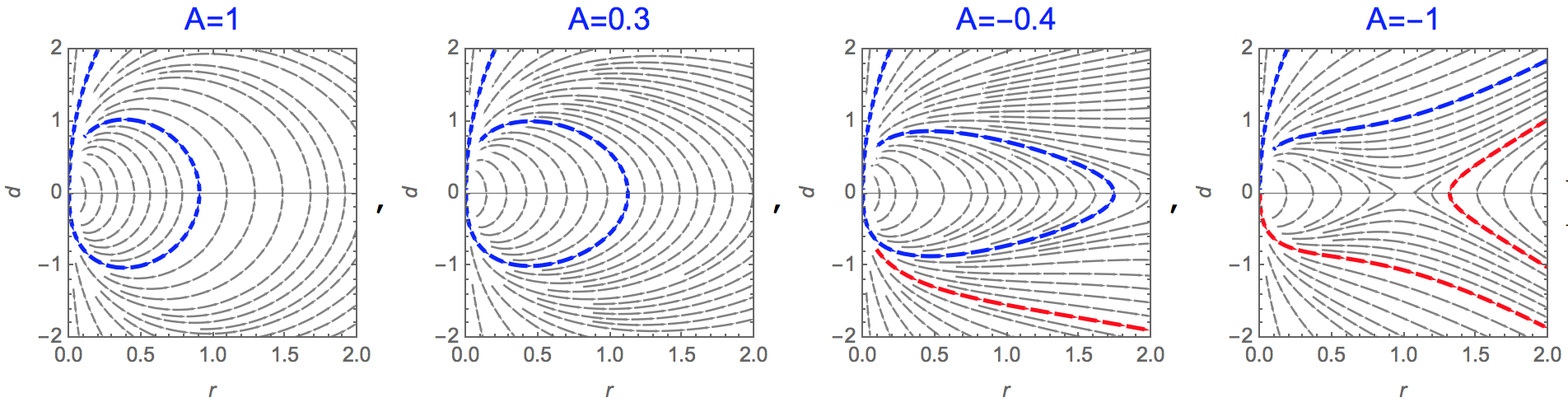}
  \caption{$c_b=0$ and $k=1$}
  \label{fig_att:sfig3}
\end{subfigure}
\caption{Numerical approximation of trajectories in the phase plane of $(\rho, d)$ with various $A$ values  for system \eqref{ode2_intro}.  For each fixed $A$ values,  blue trajectories converge to an equilibrium point, orange trajectories are oscillating solutions, and red trajectories are finite time blow-up solutions.  Top: Attractive forcing with non-zero back-ground case.  Middle: Attractive forcing with zero back-ground case. Bottom: Repulsive forcing with zero back-ground case.}\label{fig111}
\end{figure}

\section{Proof of theorem \ref{thm_poly}}
We start this section by considering the following nonlinear ODE system with the time dependent coefficient,
\begin{equation}\label{ode1_poly}
\left\{
  \begin{array}{ll}
    \dot{b} = -b^2 /2 -(t+1)^s a^2 - a , \\
    \dot{a} = -ba.\\
  \end{array}
\right.
\end{equation}

Setting $B(t)=t+1$, one can rewrite the system as follows:
\begin{equation}\label{3by3_system}
\left\{
  \begin{array}{ll}
    \dot{b} = -\frac{1}{2}b^2 - B^s a^2 -a , \\
    \dot{a} = -ba\\
    \dot{B} =1
  \end{array}
\right.
\end{equation}
with $(a , b , B)\big{|}_{t=0} = (a_0 , b_0 , B_0 =1).$

We shall find a set of initial data for which the solution of \eqref{3by3_system} exists for all time. Consider surface
$$b=m(B-1) a + n(B-1), \ B\geq1$$
in $(a,b, B)$ space where $m(\cdot)$ and $n(\cdot)$ are positive on $[0,\infty$) and continuously differentiable.
We find conditions on $m(\cdot)$ and $n(\cdot)$ such that trajectory $(a , b, B)$ stays on one side of  $3$-dimensional  surface
\begin{equation}\label{f_surface}
F(a , b , B):=b - m(B-1) a - n(B-1)=0.
\end{equation}
In order to do that, it requries
\begin{equation}\label{dot_product}
\langle \dot{a}, \dot{b}, \dot{B} \rangle \cdot \nabla F >0,
\end{equation}
on the surface $F(a ,b , B)=0$, where
$$\nabla F = \langle -m(B-1), 1, -m'(B-1)a - n'(B-1) \rangle.$$ 

Upon expanding \eqref{dot_product} and substituting \eqref{f_surface}, the left hand side of \eqref{dot_product} can be written as
\begin{equation*}
\begin{split}
&\big{\langle} -ba,  -\frac{1}{2}b^2 - B^s a^2 -a  , 1 \big{\rangle} \cdot \langle -m(B-1), 1, -m'(B-1)a - n'(B-1) \rangle \\
&\Rightarrow ba m(B-1) -\frac{1}{2}b^2 - B^s a^2 -a -m'(B-1)a - n'(B-1)\\
&\Rightarrow (ma +n) a m - \frac{1}{2}(ma +n)^2 - B^s a^2 - a -m' a + n'\\
&\Rightarrow \bigg{(} \frac{1}{2}m^2 - B^s \bigg{)} a^2 - (1+m')a - \frac{1}{2}n^2 - n' . 
\end{split}
\end{equation*}
Here and below $m$ and $n$ are evaluated at $B-1$. 

Thus, on the surface $F(a ,b , B)=0$, \eqref{dot_product} is equivalent to
\begin{equation}\label{qd_1}
\bigg{(} \frac{1}{2}m^2 - B^s \bigg{)} a^2 - (1+m')a - \frac{1}{2}n^2 - n'  >0.
\end{equation}
We will find $m$ and $n$ such that the above inequality holds for some set of $(a, b ,B)$. The inequality is quadratic in $a$. Assuming $ \frac{1}{2}m^2 - B^s >0$, for all $B\geq 1$, we shall acheive
$$(1+m')^2 -4 \big{(} \frac{1}{2}m^2 - B^s \big{)} \big{(}- \frac{1}{2}n^2 - n' \big{)}<0,$$
for all $B\geq 1$. Here, the left hand side of the above inequality is the discriminant of the quadratic equation in \eqref{qd_1}.
Thus, it suffices  to find $m$ and $n$ such that
$$-n' > \frac{(1+m')^2}{2(m^2 - 2B^s)} + \frac{1}{2}n^2.$$
Note that $m$ and $n$ are evaluated at $B-1$, so letting $$x=B-1,$$ and writing the above inequality in terms of $x$ gives,
\begin{equation}\label{surface_eqn}
-n'(x) > \frac{(1+m'(x))^2}{2(m^2 (x) -2(x+1)^s)} + \frac{1}{2}n^2 (x).
\end{equation}

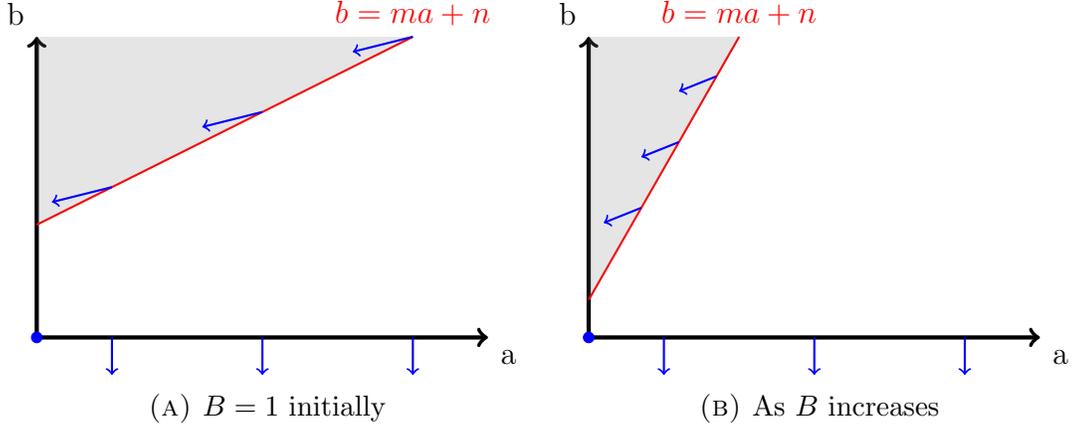
\begin{figure}[h]  
\centering 
  \begin{subfigure}[b]{0.45\linewidth}
\begin{tikzpicture}
\draw[ultra thick,->] (0,0) -- (6,0) node[anchor=north west] {a};
\draw[ultra thick,->] (0,0) -- (0,4) node[anchor=south east] {b};
\draw [thick, red] (0,1.5)  -- (5,4)  node[align=right, above]
{$b=ma + n$};
\draw [fill,blue] (0,0)  circle [radius=.07];
  \fill [gray, opacity=0.2, domain=0:2, variable=\x]
      (0, 1.5)
      --  (0, 4)
      --  (5,4)
      -- cycle;
\draw[thick, blue, ->] (1,0) -- (1,-0.5);
\draw[thick, blue, ->] (3,0) -- (3,-0.5);
\draw[thick, blue, ->] (5,0) -- (5,-0.5);
\draw[thick, blue, ->] (1, 2.5*1/5 + 1.5) -- (1 -0.8, 2.5*1/5  + 1.5 -0.2) ;
\draw[thick, blue, ->] (3, 2.5*3/5 + 1.5) -- (3 -0.8, 2.5*3/5  + 1.5 -0.2) ;
\draw[thick, blue, ->] (5, 2.5*5/5 + 1.5) -- (5 -0.8, 2.5*5/5  + 1.5 -0.2) ;
\end{tikzpicture}
    \caption{$B=1$ initially} \label{fig:M1}  
  \end{subfigure} 
\begin{subfigure}[b]{0.45\linewidth}
\begin{tikzpicture}
\draw[ultra thick,->] (0,0) -- (6,0) node[anchor=north west] {a};
\draw[ultra thick,->] (0,0) -- (0,4) node[anchor=south east] {b};
\draw [thick, red] (0,0.5)  -- (2,4)  node[align=right, above]
{$b=ma + n$};
\draw [fill,blue] (0,0)  circle [radius=.07];
  \fill [gray, opacity=0.2, domain=0:2, variable=\x]
      (0, 0.5)
      --  (0, 4)
      --  (2,4)
      -- cycle;
\draw[thick, blue, ->] (1,0) -- (1,-0.5);
\draw[thick, blue, ->] (3,0) -- (3,-0.5);
\draw[thick, blue, ->] (5,0) -- (5,-0.5);
\draw[thick, blue, ->] (0.7, 3.5*0.7/2 + 0.5) -- (0.7-0.5, 3.5*0.7/2 + 0.5-0.2);
\draw[thick, blue, ->] (1.2, 3.5*1.2/2 + 0.5) -- (1.2-0.5, 3.5*1.2/2 + 0.5-0.2) ;
\draw[thick, blue, ->] (1.7, 3.5*1.7/2 + 0.5) -- (1.7-0.5, 3.5*1.7/2 + 0.5-0.2) ;
\end{tikzpicture}

\caption{As $B$ increases} \label{fig:M2}  
\end{subfigure}
\caption{Cross-section of the ``shrinking" invariant space}\label{fig1}
\end{figure}

\textbf{Construction of $m(x)$ and $n(x)$}.   We prove the existences of $m(x)$ and $n(x)$. More precisely, we find simple polynomials 
$$n(x):=n_1 (x+n_2)^{-N}, \ \  and \ \ m(x)=m_1 (x+m_2)^M$$
that satisfy \eqref{surface_eqn} and  $m^2 (x) -2(x+1)^s >0$ for all $x\geq 0$.  We want to emphasize the method and not the technicalities, so our construction here may not be optimal, and one may obtain sharper functions $n(x)$ and $m(x)$ later.

\begin{lemma}\label{lemma_onsurface} 
If $m_1 >\sqrt{2}$, $m_2 > \max\{n_2 ,1 \}$, $M\geq \max \{ \frac{s}{2},1 \}$, $N=1$ and 
\begin{equation}\label{lemma_key_ineq}
n_1 > \max{ \bigg{\{}  \frac{(m_2 + m_1 M m^M _2)^2}{2(m^2 _1 -2)m^{2M}_2} + \frac{1}{2}, 1  \bigg{\}} }
\end{equation}
then
\begin{equation*}
-n'(x) > \frac{(1+m'(x))^2}{2(m^2 (x) -2(x+1)^s)} + \frac{1}{2}n^2 (x),
\end{equation*}
for all $x\geq 0$. Thus \eqref{surface_eqn} holds.
\end{lemma}

\begin{proof}
We show that
$$-n' > \frac{(1+m')^2}{2(m^2 -2(x+1)^s)} + \frac{1}{2}n^2 .$$
Substituting $n(x)=n_1 (x+n_2)^{-N}$ and $m(x) = m_1 (x+m_2)^M$, we consider
\begin{equation}\label{ineq_lemma_11}
n_1 N > \frac{ \{1+m_1 M (x+m_2)^{M-1} \}^2 (x+n_2)^{N+1} }{ 2\{ m^2 _1 (x+m_2)^{2M} -2(x+1)^s \} } + \frac{1}{2}\cdot \frac{1}{n^2 _1 (x+n_2)^{N-1}}
\end{equation}
Since the left hand side of the inequality is a constant, the right hand side must be non-increasing in $x$. So considering the dominating terms, this is achieved only when $N=1$.

Using $m_2 > n_2$, we bound the numerator in the right hand side of \eqref{ineq_lemma_11} first;
\begin{equation}\label{RHS_T}
 \{1+m_1 M (x+m_2)^{M-1} \}^2 (x+n_2)^{2} < \{ (x+m_2) + m_1 M (x+m_2)^M \}^2
\end{equation}
The denominator in \eqref{ineq_lemma_11} can  be bounded below;
\begin{equation}\label{RHS_B}
\begin{split}
2\{ m^2 _1 (x+m_2)^{2M} -2(x+1)^s \} &> 2 \{m^2 _1 (x+m_2)^{2M} - 2(x+m_2)^{2M} \}\\
&= 2(m^2 _1 -2 ) (x+m_2)^{2M} >0.
\end{split}
\end{equation}
Here, $m_2>1$, $M>s/2$ and $m_1 > \sqrt{2}$ are used to establish the inequalities.

Since $N=1$, from \eqref{ineq_lemma_11}, \eqref{RHS_T} and \eqref{RHS_B}, we see that it suffices to prove
$$n_1 > \frac{\{ (x+m_2) + m_1 M (x+m_2)^M \}^2}{ 2(m^2 _1 -2 ) (x+m_2)^{2M}}+ \frac{1}{2n^2 _1}.$$
We note that the right hand side of the above inequality is non-increasing in $x$ and the left hand side is a constant. Thus, it suffices to hold at $x = 0$, that is,
$$n_1 > \frac{(m_2 + m_1 M m^M _2)^2}{2(m^2 _1 -2)m^{2M}_2} + \frac{1}{2n^2 _1}.$$
Since $1/(2n^2 _1) < 1/2$ when $n_1 >1$, \eqref{lemma_key_ineq} implies the above inequality. This completes the proof.
\end{proof}

We see that the surface $F(a,b,B)=b-m(B-1)a-n(B-1)=0$ determines the invariant space of solutions $(a(t), b(t), B(t))$ for system \eqref{3by3_system}. More precisely, let
$$\tilde{\Omega}:=\{(a,b,B) \in \mathbb{R}^3 \ | a>0, F(a,b,B)>0, B\geq 1 \}.$$
Then,  all trajectories start from inside this space  will stay encompassed at all time. Furthermore, since $m(\cdot)$ and $n(\cdot)$ are positive, we note that $\tilde{\Omega}$ is located above the $aB-$plane.  See Figure \ref{fig1}. 

For system \eqref{ode1_poly}, these properties can be summarized as follows.
\begin{lemma}\label{lemma_invariant}
Consider \eqref{ode1_poly}. Let $\Omega:=\{(a ,b) \in \mathbb{R}^2  |  a>0 ,  and \  b > m(0) a + n(0)      \}$.
If $(a_0 , b_0) \in \Omega$, then $0<b(t)$ and $a(t) \leq a_0$ for all $t\geq 0$.
\end{lemma}
\begin{proof}
Solving the second equation of \eqref{ode1_poly} gives
\begin{equation}\label{a_eqn}
a(t)=a_0 e^{-\int^t _0 b(\tau) \, d \tau},
\end{equation}
and this implies that if $a_0 >0$, then $a(t)>0$, $t\geq 0$.  Together with this, Lemma \ref{lemma_onsurface} gives $b(t)>0$, $t\geq 0$.  Next, $a(t) \leq a_0$ is obtained from the positivity of $b(t)$ for all $t\geq 0$ and \eqref{a_eqn}.
\end{proof}

Now, the final step of the proof is to compare
\begin{equation}\label{comp_epsystem}
\left\{
  \begin{array}{ll}
    \dot{d} = -d^2/2 +A(t) \rho^2 +k\rho, \\
    \dot{\rho} = -d\rho\\
  \end{array}
\right.
\end{equation}
with
\begin{equation}\label{comp_auxsystem}
\left\{
  \begin{array}{ll}
    \dot{b} = -b^2 /2 -(t+1)^s a^2 -a , \\
    \dot{a} = -ba.\\
  \end{array}
\right.
\end{equation}
Here, $k$ is either $1$ or $-1$.
We recall that
$$-(t+1)^s \leq A(t) \leq  \frac{1}{2}\bigg{(} \frac{\omega_0}{\rho_0} \bigg{)}^2, \ \ t\geq 0.$$
We have the monotonicity relation between two ODE systems. We should point out that auxiliary system \eqref{comp_auxsystem} serves as a comparison system for \emph{both} repulsive and attractive cases.
\begin{lemma}\label{lemma_comp}
Consider \eqref{comp_epsystem} and \eqref{comp_auxsystem}, with either $k=1$ or $k=-1$. Then
\begin{equation*}
\left\{
  \begin{array}{ll}
	b(0) < d(0), \\
    0<\rho(0)<a(0)\\
  \end{array}
\right.
implies
\ 
\left\{
  \begin{array}{ll}
	b(t) < d(t), \\
    0<\rho(t)<a(t)\\
  \end{array}
\right.
\ for \ all \ t>0.
\end{equation*}
\end{lemma}

\begin{proof}
Suppose $t_1$ is the earliest time when the above assertion is violated. Consider
\begin{equation}\label{comp_a_r}
a(t_1) = a(0) e^{-\int^{t_1} _ 0 b(\tau) \, d \tau} > \rho(0) e^{-\int^{t_1} _0 d(\tau) \, d \tau} = \rho (t_1).
\end{equation}
Therefore, it is left with only one possibility that $d(t_1) = b(t_1).$ Consider
\begin{equation}\label{diff_systems}
\dot{b}-\dot{d} =-\frac{1}{2} (b^2 -d^2) - (t+1)^s a^2 - A(t) \rho^2 - a - k\rho.
\end{equation}
Since $b(t)-d(t) <0$ for $t<t_1$ and $b(t_1) - d(t_1)=0$, hence at $t=t_1$, we have
$$\dot{b}(t_1)-\dot{d} (t_1) \geq 0.$$
But the right hand side of \eqref{diff_systems},  when it is evaluated at $t=t_1$, is negative. Indeed
\begin{equation*}
\begin{split}
&-\frac{1}{2} \big{(}b^2 (t_1) -d^2 (t_1) \big{)} - (t_1 +1)^s a^2 (t_1) - A(t_1) \rho^2 (t_1) - a(t_1) - k\rho(t_1)\\
&= - (t_1 +1)^s a^2 (t_1) - A(t_1) \rho^2 (t_1) - a(t_1) - k\rho(t_1)\\
&=(t_1 +1)^s \big{(}  -a^2 (t_1) + \rho^2 (t_1) \big{)} +\rho^2 (t_1) \big{(} -(t_1+1)^s -A(t_1)  \big{)} - a(t_1) - k\rho(t_1).
\end{split}
\end{equation*}
From \eqref{comp_a_r}, we see that $-a(t_1) + \rho (t_1) <0$. Also, $-a(t_1) - \rho(t_1) <0$ because of positivities of $a(0)$ and $\rho(0)$. Thus, $-a(t_1) -k\rho(t_1) <0$ for both $k=1$ and $k=-1$.
Also,  $-(t_1 +1)^s \leq A(t_1)$. Therefore, the right hand side of \eqref{diff_systems} at $t=t_1$ is negative, and this leads to the contradiction.
\end{proof}

\begin{lemma}\label{lemma_dbound}
 Consider \eqref{comp_epsystem} with either $k=1$ or $k=-1$. If there exists $\rho_M  >0$ such that $\rho(t) \leq \rho_M$, $\forall t \geq 0$, then $d(t)$ is bounded from above for all $d_0 \in \mathbb{R}$.
\end{lemma}
\begin{proof}
Since $A(t)\leq   \frac{1}{2}( \frac{\omega_0}{\rho_0} )^2 =:w$, we have
\begin{equation*}
\begin{split}
\dot{d}&= -d^2 /2 + A(t)\rho^2 +k \rho \\
&\leq  -d^2 /2  +    w \rho^2 +k \rho \\
&\leq  -d^2 /2  + \max \{ 0,     w \rho^2 _M +k \rho_M  \}.
\end{split}
\end{equation*}
Thus,
$$d(t) \leq \max \big{\{} d_0, \sqrt{2\max\{ 0, w\rho^2 _M +k \rho_M \} }  \big{\}}.$$
\end{proof}

The last step of proving the theorem is to combine the comparison principle in Lemma \ref{lemma_comp} with Lemma \ref{lemma_invariant}.
Note that $\Omega$ defined in Lemma \ref{lemma_invariant} is  an open set and given any initial data $(\rho_0 , d_0) \in \Omega$ for system \ref{comp_epsystem}, we can find $\epsilon>0$ and initial data $(a_0 , b_0):=(\rho_0 +\epsilon , d_0 -\epsilon) \in \Omega$ for system \ref{comp_auxsystem}. Therefore, by lemmata  \ref{lemma_comp} and \ref{lemma_invariant},
$$0<\rho(t) < a(t) \leq a_0, \ \ and \ \ 0< d(t), \ \ \forall t\geq 0.$$
Thus, we see that $\rho(t)< a_0$ and this implies that $d(t)$ is bounded from above for all $t\geq 0$ due to  Lemma \ref{lemma_dbound}.
This completes the proof.

$$$$
\section{Appendix}
In this appendix, as mentioned in the remark of the theorem, we discuss why the EP system
\begin{equation}
\left\{
  \begin{array}{ll}
    d' = -\frac{1}{2}d^2 + A(t)\rho^2 +k(\rho-c_b), \\
    \rho' = -\rho d,\\
  \end{array}
\right.
\end{equation}
is unable to attain global solutions when $A(t) \sim -e^{t}$ and $c_b =0$.

In \cite{Lee20}, global solutions for the EP system 
with \emph{non-zero background} case (i.e., $c_b\neq 0$) were investigated under the assumption that $A(t)\geq -(\alpha t+\beta)^s, \ \ t\geq 0.$ Furthermore, similar results were obtained under the weaker assumption
\begin{equation}\label{exp_assump_11}
A(t) \geq -\alpha e^{\beta t}, \ \ t\geq 0,
\end{equation}
where $\alpha$, $\beta$ are positive constants. That is, the EP system can afford to have global solutions as long as the decay rate of $A(t)$ is not too high - an exponential rate.

Returning to the zero background case, one should not expect the similar results to hold under the same assumption in \eqref{exp_assump_11}. Indeed, consider the following auxiliary system
\begin{equation}\label{exp_system_11}
\left\{
  \begin{array}{ll}
    b' = -b^2/2 - e^t a^2 -a, \\
    a' = -b a.\\
  \end{array}
\right.
\end{equation}
In the following lemma, one can see that finite time blow up occurs for \emph{all} initial data, which contrasts to the existence of global solutions for the auxiliary system 
\begin{equation}\label{poly_system_223}
\left\{
  \begin{array}{ll}
    \dot{b} = -b^2 /2 -(t+1)^s a^2 - a , \\
    \dot{a} = -ba.\\
  \end{array}
\right.
\end{equation}
That is, in Lemma \ref{lemma_invariant}, we show that there exists some bounded solutions (actually converge to the equilibrium point-the origin) for the system \eqref{poly_system_223}. In contrast to this, we have the following:

\begin{figure}[h]  
\centering 
\begin{tikzpicture}
\draw[ultra thick,->] (-0.4,0) -- (5,0) node[anchor=north west] {a};
\draw[ultra thick,->] (0,0) -- (0,2.5) node[anchor=south east] {b};
\draw[ultra thick,->] (0,0) -- (0,-1.5);

\draw [thick, red] (0,1) -- (5,1)  node[align=right, right]
{$b=1/2$};;

\draw [fill] (2.5, 2)  node[below]{$\Omega_T$};
\draw [fill] (2.5, 0.8)  node[below]{$\Omega_M$};
\draw [fill] (2.5, -0.5)  node[below]{$\Omega_B$};
      \draw [ draw=gray, fill=gray, fill opacity=0.2]
       (0,0) -- (0,1-0.02) -- (5-0.1,1-0.02)  --  (5-0.1,0);
       \draw [ draw=gray, fill=gray, fill opacity=0.1]
       (0,1) -- (0,2.4) -- (5-0.1,2.4)  --  (5-0.1,1);
       \draw [ draw=gray, fill=yellow, fill opacity=0.2]
       (0,0) -- (0,-1.5+0.1) -- (5-0.1, -1.5+0.1)  --  (5-0.1,0);
\end{tikzpicture}
    \caption{Decomposition of $\mathbb{R}^+ \times \mathbb{R}$ in Lemma \ref{blowup_lemma}} \label{fig:M1}  
\end{figure}
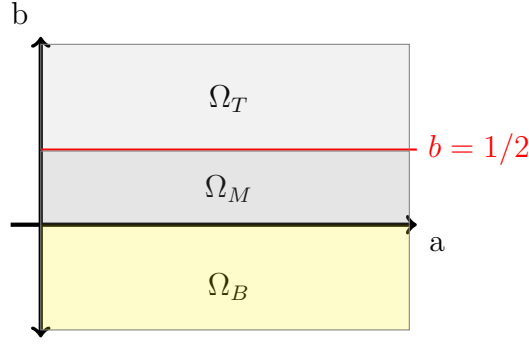

\begin{lemma}\label{blowup_lemma}
Consider \eqref{exp_system_11}. For every $(a_0 , b_0)\in (\mathbb{R}^+,  \mathbb{R})$, it holds that
$$a(t) \rightarrow \infty \ and \ b(t) \rightarrow -\infty,$$
in finite time.
\end{lemma}
\begin{proof}
We first decompose  $(\mathbb{R}^+,  \mathbb{R})$ into three subsets. See Figure \ref{fig:M1}. Let $\Omega_B = \{(a,b) | a>0, \ b < 0  \}$, $\Omega_M = \{(a,b) | a>0, \   0 \leq b < 1/2  \}$ and $\Omega_T = \{(a,b) | a>0, \   b\geq 1/2  \}$.
Since $a(t)=a_0 e^{-\int^t _0 b(\tau) \, \tau} >0$, from the first equation of \eqref{exp_system_11}, we see that $b$ is strictly decreasing.

Let $(a_0 , b_0)\in \Omega_B$. That is, $b_0 <0$. Integrating $b' = -\frac{1}{2}b^2 - e^t a^2 -a  \leq -\frac{1}{2}b^2$ gives
\begin{equation}\label{b_ineq}
b(t) \leq \frac{2}{t+2b^{-1} _0},
\end{equation}
which in turn implies $b(t) \rightarrow -\infty$, as $t\rightarrow -2/b_0.$

Next, we show that if $(a_0 , b_0)\in \Omega_M$, then $(a(t), b(t)) \in \Omega_B$ in finite time. Since $b(t) \leq b_0$ for all $t$, as long as it exists, we have
$$a(t) = a_0 e^{-\int^t _0 b(\tau) \, d \tau} \geq a_0 e^{-\int^t _0 b_0 \, d \tau}=a_0 e^{-b_0 t}.$$
This inequality gives
$$b' \leq -\frac{1}{2}b^2 - e^t a^2 _0 e^{-2b_0 t} -a \leq -a^2 _0 e^{(1-2b_0)t}.$$
Since $1-2b_0 >0$, we see that $b(t)<0$ in finite time. 

Finally, it is easy to see that  $(a_0 , b_0)\in \Omega_T$ implies  $(a(t) , b(t))\in \Omega_M$ in finite time. Indeed, using \eqref{b_ineq} and $b_0 \geq 1/2$, we see that there exists $t^* < \infty$ such that $b(t^*)<\frac{1}{2}$.
\end{proof}

\bibliographystyle{abbrv}

\end{document}